\documentclass[12pt]{article}
\usepackage{amsmath, amssymb}
\usepackage[dvipdfmx]{graphicx}
\usepackage[dvipdfmx]{color}
\usepackage{enumerate}
\usepackage[numbers,sort]{natbib}
\usepackage{comment}
\usepackage{url} 

\usepackage{amsthm}
\theoremstyle{definition}
\newtheorem{theorem}{Theorem}
\newtheorem{lemma}{Lemma}
\newtheorem{proof of lemma}{proof of Lemma}
\usepackage{here}

\newcommand{\blind}{1}

\begin{document}

\def\spacingset#1{\renewcommand{\baselinestretch}%
{#1}\small\normalsize} \spacingset{1}

\if1\blind
{
 \title{\bf Testing Homogeneity for Normal Mixture Models: Variational Bayes Approach}
 \author{Natsuki Kariya, and Sumio Watanabe\hspace{.2cm}\\
 Department of Mathematical and Computing Science, \\ Tokyo Institute of Technology\\
 }
 \maketitle
} \fi

\if0\blind
{
 \bigskip
 \bigskip
 \bigskip
 \begin{center}
 {\LARGE\bf Testing Homogeneity for Normal Mixture Models: Variational Bayes Approach}
\end{center}
 \medskip
} \fi

\bigskip

\begin{abstract}
The test of homogeneity for normal mixtures has been conducted in diverse research areas, but constructing a theory of the test of homogeneity is challenging because the parameter set for the null hypothesis corresponds to singular points in the parameter space. 
In this paper, we examine this problem from a new perspective and offer a theory of hypothesis testing for homogeneity based on a variational Bayes framework. 
In the conventional theory, the constant order term of the free energy has remained unknown, however,
we clarify its asymptotic behavior because it is necessary for constructing  
a hypothesis test.  Numerical experiments shows the validity of our theoretical results. 
\end{abstract}

\noindent%
\vfill

\newpage
\spacingset{1.45} 
\section{Introduction}
\label{sec:intro}

Mixture models are very useful for describing the data that comprises the effects of several different factors. These models have been used  in various fields, including pattern recognition, clustering analysis, and anomaly detection\cite{Mclachlan2000}. Identifying the number of the clusters that affect the data is a very important problem, and a hypothesis test is one useful tools for this purpose. This type of the tests is called testing homogeneity. Testing homogeneity has been considered for various mixture models, but its use for normal mixture models has been studied especially\cite{Chauveau2017}. 

Theoretically, mixture models often have singularity in their parameter space where the Fisher information matrix becomes singular. This results in the log likelihood ratio not converging to $\chi^{2}$ distributions, unlike the case of the regular models. This is why testing homogeneity for mixture model is theoretically challenging. \cite{Hartigan1985}\cite{Garel2001}.

To circumvent this problem, various methods have been proposed, such as the modified likelihood ratio test, a method that adds a regularizing term \cite{Chen2001}\cite{Chen2004},  a D test\cite{Charnigo2004}, and applying an expectation-maximization (EM) algorithm for calculating the modified likelihood ratio\cite{Chen2009}\cite{Chen2012}, and so on. 
However,  few studies exists that treats the problem based on a $\it{Bayesian} $ treatment.

For statistical inference using singular models, the properties and effectiveness of a Bayesian treatment have been clarified  through the learning theory\cite{Watanabe2018}. It is natural to consider the application of this theory to the problem of the hypothesis test using singular models. However, theoretical studies with such motivation are still very limited.

In the  Bayesian hypothesis test, one must calculate the test statistics, the marginal likelihood ratio, from the posterior.  In general, this procedure is costly, and an efficient method is needed. Variational Bayes (also called variational inference)\cite{Attias2000}\cite{Blei2017} is a popular and useful method for fulfilling this requirement. However, to the best of our knowledge, no studies have yet applied variational Bayes to the approximation of the marginal likelihood ratio and used it to construct a hypothesis test has not been studied, especially for testing homogeneity, as far as the authors know.

There have been  some studies on the asymptote of the variational free energy for mixture models\cite{WatanabeWatanabe2006}\cite{WatanabeWatanabe2007}. One important results of these studies is that the phase transition is induced by the hyperparameter. Phase transitions drastically change the behavior of the test statistics, and the properties of the phase transitions must be studied to constructing a hypothesis test.

Also, one must obtain the stochastic behavior of the test statistics for constructing a Bayesian hypothesis test. Previous work has already shown that the stochastic term of the logarithm of the test statistics (variational free energy) is  $\mathcal{O}(1)$, while the leading term is $\mathcal{O}(\log n)$\cite{WatanabeWatanabe2006}\, but in the previous work, the estimation of the variational free energy is within the order of $\mathcal{O}(\log n)$. Therefore, the estimation of the variational free energy  to a higher order  is needed, but this has not been accomplished yet.

In this paper, we theoretically construct a new way to test for homogeneity of normal mixture models based on the variational Bayes framework. 
We apply the scheme of the variational Bayes to our problem and theoretically derive the asymptotic distribution of the constant order term in the variational free energy, a task accomplished for the first time to the best of our knowledge. We show that our model has the phase transition, and we clarify the critical value.  Also we derive the analytical expression of the variational free energy on the order of $\mathcal{O}(1)$, when the hyperparameter $\phi$ is larger than the critical value. We construct a new hypothesis test based on our results and demonstrate its validity with numerical experiments.

\section{Bayesian hypothesis test}

In this section, we briefly review the framework of a Bayesian hypothesis test. We also define our problem concretely and introduce the latent variables.

Let $\left\{X^{n} =(X_1,X_2,...,X_n)\in \mathbb{R}^{1}\right\}$ be sample,  generated independently and identically from a probabilistic model $p_{0}(x|w)$, 
\begin{equation}
p_{0}(x|w) = (1-a) \mathcal{N}(0,1^2) + a\mathcal{N}(b,1^2), 
\end{equation}
where $a$ and ${\cal N}(b,1^2)$ show the mixture ratio and the normal distribution whose average is $b$ and whose variance is $1^{2}$, respectively. 
The parameter of this model is $w=(a,b)$, where $0\leq a \leq 1$ and $b\in\mathbb{R}$. 

In the Bayesian framework,  parameters $w_{0}$ is assumed to be generated from a prior $\varphi(w)$, which is described as 
\[
w_0\sim \varphi(w),\;\;\;X_i\sim p_{0}(x|w_0). 
\]
For testing homogeneity in a normal mixture model, the null and alternative hypotheses are set as
\begin{eqnarray*}
\mbox{N.H.} & : & w_0\sim \varphi_0(w),\;\;\;X_i\sim p_{0}(x|w_0), \\
\mbox{A.H.} & : & w_0\sim \varphi_1(w),\;\;\;X_i\sim p_{0}(x|w_0). 
\end{eqnarray*}
The marginal likelihood ratio can be written as,
\begin{equation}
L(X^{n}) = \frac{\displaystyle \int \varphi_{1}(w) \prod_{i}p_{0}(X_{i}|w) dw}
{\displaystyle \int \varphi_{0}(w) \prod_{i}p_{0}(X_{i}|w) dw}. 
\end{equation}
In this paper, we discuss the asymptotic properties of $L(X^n)$ for the hypothesis test of homogeneity. 

We assume that the N.H and A.H are as follows,
\begin{eqnarray*}
\varphi_0(a,b) &=& \delta(a)\delta(b), \\
\varphi_{1}(a,b) &=& U_{a}(0,1) \times \frac{1}{\sqrt{2\pi\sigma^{2}}}\exp{\left(-\frac{1}{2\sigma^{2}}b^{2}\right)}.
\end{eqnarray*}
where $U_a(0,1)$ is a uniform distributions of $a$ on $(0,1)$. 

The essentially difficult task is calculating the numerator of $L(X^{n})$. This is equivalent to the integration of the posterior in the parameter space, under the A. H. In the following sections, we will discuss how to estimate this quantity and construct the hypothesis test based on it.

First, we introduce latent variables $\left\{y_{ik}\right\}$ for convenience. The variable $y_{ik} \in \left\{0,1\right\}$ shows to which cluster in the probabilistic model the sample ${X_{i}}$ belongs.  Note that the latent variables satisfy the relation $\sum_{k}y_{ik} = 1$.

Using the latent variables, the posterior under the A. H.  can be written as,
\begin{equation}
p(w,\left\{y_{ik}\right\}|X^{n}) \equiv \frac{1}{Z_{n}}\prod_{k}\left\{a_{k}e^{-\left(X_{i}-b_{k}\right)^{2}/2}\right\}^{y_{ik}}\varphi_{1}(w) 
\end{equation}
where $a_{0} = (1-a)$ and $a_{1} = a$, and $b_{0} = 0$ and $b_{1} = b$. We simply write the set  of the  parameter$\left\{a,b\right\}$ as $w$, and the summation of $\left\{y_{ik}\right\}$ is taken for all configurations, $Z_{n}$ is

\[
Z_{n} \equiv \displaystyle \int dw \sum_{\left\{y_{ik}\right\}}\prod_{k}\left\{a_{k}e^{-\left(X_{i}-b_{k}\right)^{2}/2}\right\}^{y_{ik}}\varphi_{1}(w).
\]

\newpage

\section{Variational approximation for conditional probability $p(w|X^{n})$}
In this section, we approximate $p(w,\left\{y_{ik}\right\}, X^{n})$ using  the variational Bayes approach. That is, we derive a function $q(\left\{y_{ik}\right\})r(w)$ that minimizes the Kullbuck-Leibler divergence,

\begin{equation*}
D (qr||p)= \int dw \sum_{\left\{y_{ik}\right\}} q(\left\{y_{ik}\right\})r(w) \log{\frac{q(\left\{y_{ik}\right\})r(w)}{p(w,X^{n})}}.
\end{equation*}
The $q(\left\{y_{ik}\right\})$ and $r(w)$ that minimize the Kullbuck-Leibler divergence should satisfy the following conditions, which are derived from the variational principle.
\[
q(\left\{y_{ik}\right\}) \propto \exp{\left[E_{r}\left\{\log {p(w|X^{n})}\right\}\right]}, 
\]
\[
r(w) \propto \exp{\left[E_{q}\left\{\log {p(w|X^{n})}\right\}\right]}.
\]
where $E_{r}\left\{\cdot \right\}$ means the expected value with respect to $r(w)$, and $E_{q}\left\{\cdot \right\}$ means the expected value with respect to $q(\left\{y_{ik}\right\})$.
The logarithm of $p(w,\left\{y_{ik}\right\}, X^{n}) $ becomes
\begin{eqnarray}
\log{p(w,\left\{y_{ik}\right\},X^{n}) } &=&  \sum_{i}\sum_{k}y_{ik}\left[\log{a_{k}}  - \frac{1}{2}\left(X_{i}-\delta_{k1}b\right)^{2}\right] \nonumber \\
&-&n\log{\left(2\pi\sigma^{2}\right)}+\log{\varphi_{1}(a,b)}, 
\end{eqnarray}
It is linear with respect to $y_{ik}$, and we can easily derive $r(w)$, 
\begin{eqnarray}
r(w) &\propto& \exp{\left[E_{q}\left\{\log{p(w,\left\{y_{ik}\right\},X^{n}) }\right\}\right]} \nonumber \\ 
&=& \prod_{k}\prod_{i}a_{k}^{\hat{y_{ik}}} \frac{1}{\sqrt{2\pi}}\left\{\exp{\left[-\frac{1}{2}\left(X_{i}-\delta_{k1}b\right)^{2}\right]}\right\}^{\hat{y_{ik}}} \nonumber \\
&\times& \frac{1}{\sqrt{2\pi\sigma^{2}}}\exp{\left[-\frac{1}{2\sigma^{2}}b^{2}\right]},
\end{eqnarray}
where $\hat{y_{ik}}$ means $E_{q}\left\{y_{ik}\right\}$.

Similar to $r(w)$, we can derive $q(\left\{y_{ik}\right\})$, which becomes

\begin{eqnarray}
q(y_{ik})  &\propto& \exp{\left[\sum_{i}\sum_{k}y_{ik}\left\{\langle \log{a_{k}}\rangle -\frac{1}{2}\langle\left(X_{i}-\delta_{k1}b\right)^{2}\rangle\right\}\right]} \nonumber \\
&=& \prod_{i}\prod_{k}\left\{\exp{\left[\langle \log{a_{k}}\rangle - \frac{1}{2}\langle \left(X_{i}-\delta_{k1}b\right)^{2}\rangle \right]}\right\}^{y_{ik}}
\end{eqnarray}
where $\langle\left\{\cdot\right\}\rangle$ is abbreviation of $E_{r}\left\{\cdot \right\}$. 

The self-consistent equation that $\hat{y_{ik}}$ should satisfy becomes
\begin{equation}
\hat{y_{ik}} \propto \exp{\left[\langle \log{a_{k}}\rangle - \frac{1}{2}\langle \left(X_{i}-\delta_{k1}b\right)^{2}\rangle\right]}
\end{equation}
In addition, the self-consistent equation $a_{k}$ should satisfy is
\begin{equation}
\langle \log{a_{k}}\rangle = \psi \left(\sum_{i}\hat{y_{ik}} + 1\right) - \psi\left(n + 2\right),
\end{equation}
where $\psi(\cdot)$ is the digamma function.

From these results, $r(w)$ can be written as,
\begin{eqnarray}
r(w) &\propto& \prod_{k}\prod_{i}a_{k}^{\hat{y_{ik}}}\exp{\left[-\frac{1}{2}\left(\frac{1}{\sigma^{2}}+\sum_{i}\hat{y_{ik}}\right)\right]} \nonumber \\
 &\times& \exp{\left(b-\frac{\sum_{i}X_{i}\hat{y_{ik}}}{\frac{1}{\sigma^{2}}+\sum_{i}\hat{y_{ik}}}\right)^{2}}
\end{eqnarray}
and $\langle b\rangle$ and $\langle b^{2} \rangle$ should satisfy
\[
\langle b\rangle = \frac{\sum_{i}X_{i}\hat{y_{ik}}}{\sum_{i}\hat{y_{ik}}+\frac{1}{\sigma^{2}}},
\]

\[
\langle b^{2}\rangle = \langle b\rangle^{2} + \frac{1}{\sum_{i}\hat{y_{ik}}+\frac{1}{\sigma^{2}}}.
\]
We obtain the self-consistent equations for $\hat{y_{ik}}$ as follows,

\begin{eqnarray}
\hat{y_{i0}} &\propto& \exp{\left[\psi\left(\sum_{i}\hat{y_{i0}}+1\right)-\psi(n+2) -\frac{1}{2}X_{i}^{2}\right]} \\ \nonumber 
\hat{y_{i1}} &\propto& \exp{\left[\psi\left(\sum_{i}\hat{y_{i1}}+1\right)-\psi(n+2) - \frac{1}{2}\langle \left(X_{i}-b\right)^{2}\rangle\right]} \\ \nonumber 
&=& \exp{\left[\psi\left(\sum_{i}\hat{y_{i1}}+1\right)-\psi(n+2) \right]} \\
&\times& \exp{\left[- \frac{1}{2}\left\{\left(X_{i}-\langle b\rangle \right)^{2} +\frac{1}{\sum_{i}\hat{y_{i1}}+\frac{1}{\sigma^{2}}}\right\} \right]} 
\end{eqnarray}
The variational free energy becomes
\begin{eqnarray}
F &=& \langle q(\left\{y_{ik}\right\})r(w) \log{\frac{q(\left\{y_{ik}\right\})r(w)}{p(w,X^{n})}} \rangle \nonumber \\ 
&=& \sum_{i}\sum_{k}\hat{y_{ik}}\log{\hat{y_{ik}}} + \log{\frac{\Gamma(\sum \langle a_{k}\rangle)}{\prod \Gamma(\langle a_{k}\rangle)}} \nonumber \\
&+&\frac{1}{2}\log{\left(1+\sigma^{2}\sum_{i}\hat{y_{i1}}\right)} + \frac{1}{2\sigma^{2}}\langle b\rangle^{2} \nonumber \\ 
&+& \frac{1}{2}\sum_{i}\sum_{k}\hat{y_{ik}}\left(X_{i}-\delta_{k1}\langle b_{k}\rangle\right)^{2} + \frac{1}{2}\sum_{i}\sum_{k}\hat{y_{ik}}\log{\left(2\pi\right)} \nonumber \\
&=& \sum_{i}\sum_{k}\hat{y_{ik}}\log{\hat{y_{ik}}} + \log{\frac{\Gamma(\sum_{k} \sum_{i}\hat{y_{ik}}+1)}{\prod \Gamma(\sum_{i}\hat{y_{ik}}+1)}}  \nonumber \\
&+& \frac{1}{2}\log{\left(1+\sigma^{2}\sum_{i}\hat{y_{i1}}\right)} + \frac{1}{2}\sum X_{i}^{2} \nonumber \\
&-& \frac{1}{2}\frac{\left(\sum X_{j}\hat{y_{j1}}\right)^{2}}{\sum_{i}\hat{y_{i1}}+\frac{1}{\sigma^{2}}} + \frac{n}{2} \log{\left(2\pi\right)}
\end{eqnarray}
The logarithm of the denominator of $L$ is calculated as 

\begin{eqnarray*}
F_{0} &=& - \log{\int  \varphi_{0}(w) \prod p_{0}(X_{i},w) dw } \nonumber \\
&=& \frac{1}{2}\sum_{i} X_{i}^{2} + \frac{n}{2}\log{\left(2\pi\right)}
\end{eqnarray*}
Therefore, we obtain the logarithm of $L$,  

\begin{eqnarray}
F - F_{0} &=& \sum_{i}\sum_{k}\hat{y_{ik}}\log{\hat{y_{ik}}} +  \log{\frac{\Gamma(\sum_{k} \sum_{i}\hat{y_{ik}}+1)}{\prod \Gamma(\sum_{i}\hat{y_{ik}}+1)}} \nonumber \\
&+& \frac{1}{2}\log{\left(1+\sigma^{2}\sum_{i}\hat{y_{i1}}\right)}  - \frac{1}{2}\frac{\left(\sum X_{j}\hat{y_{j1}}\right)^{2}}{\sum_{i}\hat{y_{i1}}+\frac{1}{\sigma^{2}}} 
\end{eqnarray}
We can extend the result above and obtain the variational free energy when the prior of $a$ is  a Dirichlet distribution, $\varphi(a) \propto (1-a)^{\phi-1}a^{{\phi-1}}$. The result is 

\begin{eqnarray}
F - F_{0} &=& \sum_{i}\sum_{k}\hat{y_{ik}}\log{\hat{y_{ik}}} +  \log{\frac{\Gamma(\sum_{k} \sum_{i}\hat{y_{ik}}+2\phi)}{\prod \Gamma(\sum_{i}\hat{y_{ik}}+\phi)}} \nonumber \\
&+& \frac{1}{2}\log{\left(1+\sigma^{2}\sum_{i}\hat{y_{i1}}\right)}  - \frac{1}{2}\frac{\left(\sum X_{j}\hat{y_{j1}}\right)^{2}}{\sum_{i}\hat{y_{i1}}+\frac{1}{\sigma^{2}}} \nonumber \\
&-& \log{\frac{\Gamma(2\phi)}{\prod(\Gamma(\phi))}}
\end{eqnarray}
If we can derive the asymptotic distribution of $F - F_{0}$, we can construct a hypothesis test. This requires the stochastic behavior of $\hat{y_{i1}}$.  However, as discussed in the next section, the variational free energy exhibits the phase transition when the hyperparameter $\phi$ changes, This affects the configuration and stochastic behavior of $\hat{y_{i1}}$.

\section{Phase transition induced by the hyperparameter}
In our problem, the candidates for the parameter sets that minimize the variational free energy are those that corresponding to the null hypothesis.  However, the parameter sets that corresponds to the null hypothesis are not unique. Specifically, $\left\{y_{i1}\right\}$ that satisfies $\sum_{y_{i1}} = \mathcal{O}(1)$ and $\langle b \rangle = 0$ is one candidate, but also $\left\{y_{i1}\right\}$ that satisfies $\sum y_{i1} = 0$ is another one. 

In a previous study treating normal mixtures\cite{WatanabeWatanabe2006},  the upper and lower bounds of the asymptote of the variational free energy were derived within $\mathcal{O}(\log(n))$, and the existence of the phase transition induced by the hyperparameter was proven. We can expect that the phase transition to occur in our model as well, and it should be examined.

The phase transition affects the stochastic behavior of the variational free energy, the test statistics. Therefore, we must study the effect of the phase transition and grasp what kind of configuration is obtained as a function of the hyperparameter.This is the main purpose in this section. We firstly show the existence of the phase transition and derive the critical point $\phi_{\rm{cr}}$.

\newpage 

\subsection{Asymptotic form of $F$ when $\sum_{i}y_{i1}$ is $\mathcal{O}(n)$}
Our purpose is to construct the hypothesis test, and we focus on a situation when the hypothesis test is important, namely, one in which distinguishing two hypotheses is difficult.  Specifically,  when $\langle b \rangle$ is small and two gaussian distribution in the model are largely overlapped, distinguishing the two distributions is difficult. We assume such a situation, specifically $\langle b\rangle X_{\rm{max}} \sim o_{p}(1)$, under the null hypothesis. 

Under these conditions,  the following theorem holds.
\begin{theorem}
Under $\sum_{i}y_{i1}$ is $\mathcal{O}(n)$ and $\langle b\rangle \sim o(1/\sqrt{\log n})$, the asymptotic form of the variational free energy becomes
\begin{equation}
F - F_{0} =  \log{n} + o(\log n)
\end{equation}
under the null hypothesis.
\end{theorem}

\begin{proof}
Let us introduce  $\bar{y}$ as $\sum_{i}\hat{y_{i1}} \equiv n_{1}$ and $n_{1}/n \equiv \alpha \sim \mathcal{O}(1)$ for brevity.
The self-consistent equation of $\left\{y_{i1}\right\}$becomes
\begin{eqnarray*}
\hat{y_{i1}} &=& \frac{n_{1}e^{\langle b \rangle X_{i}-1/2\langle b^{2\rangle}}}{n-n_{1} + n_{1} e^{\langle b \rangle X_{i}-1/2\langle b^{2\rangle}}} \\
&=&  \frac{n_{1}}{n} + \frac{n_{1}}{n}\left(1-\frac{n_{1}}{n}\right) \langle b\rangle X_{i}  \\
&+& \frac{1}{2}\frac{n_{1}}{n}\langle b\rangle^{2}\left(1-\frac{n_{1}}{n}\right) \left(X_{i}^{2}-1\right) \\
&-& \left(\frac{n_{1}}{n}\right)^{2}\langle b\rangle^{2} X_{i}^{2}+ \left(\frac{n_{1}}{n}\right)^{3}\langle b\rangle^{2} X_{i}^{2} + \mathcal{O}(\langle b\rangle^{3} )\\
&=& \alpha + \alpha (1-\alpha) \langle b \rangle X_{i} \\
&+& \frac{1}{2}\alpha \langle b\rangle^{2}\left[X_{i}^{2}\left(1-3\alpha + 2\alpha^{2}\right) + \alpha-1\right] + \mathcal{O}(\langle b\rangle^{3} )
 \end{eqnarray*} 
Using this expression, $\langle b\rangle $ becomes 
\begin{eqnarray*}
\langle b\rangle &=& \frac{\sum_{j}\hat{y_{j1}}X_{j}}{\sum_{j}\hat{y_{j1}}+\frac{1}{\sigma^{2}}} \nonumber \\
&=& \frac{\sum_{i}X_{i}\left(\alpha + \left(1-\alpha\right)\langle b\rangle X_{i} + \mathcal{O}(\langle b\rangle^{2})\right)}{n_{1}+\frac{1}{\sigma^{2}}} \\
&=&\frac{1}{n}\sum X_{i} + \langle b\rangle \left(1-\alpha\right) 
\end{eqnarray*}
Therefore, we obtain 
\begin{equation}
\langle b\rangle =\frac{1}{n_{1}}\sum X_{i} = \mathcal{O}(\frac{1}{\sqrt{n}})
\end{equation}
$\hat{y_{i1}}$ becomes
\begin{eqnarray*}
\hat{y_{i1}} &=& \alpha + \frac{(1-\alpha)}{n}\sum_{j}X_{j} X_{i} \nonumber \\
&+& \frac{1}{2\alpha n^{2}}\left(\sum X_{j}\right)^{2}\left[\alpha -1 + \left(1-3\alpha + 2\alpha^{2}\right)X_{i}^{2}\right]
\end{eqnarray*}
Let us calculate the variational free energy. For simplicity,  we write $\hat{y_{i1}}$ as 
\begin{equation*}
\hat{y_{i1}} = \alpha + \Delta y_{i}
\end{equation*}
The entropy term becomes, 
\begin{eqnarray*}
& &\sum_{i}\left\{\hat{y_{i1}}\log{\hat{y_{i1}}} + (1-\hat{y_{i1}})\log{(1-\hat{y_{i1}})}\right\} \\
&=& \sum_{i}\left(\alpha + \Delta y_{i}\right)\log \left(\alpha + \Delta y_{i}\right) \\
&+& \left[1 - \left(\alpha + \Delta y_{i}\right)\right] \log \left[1 - \left(\alpha + \Delta y_{i}\right)\right] \\
&=& n\left[\alpha \log \alpha + \left(1-\alpha\right) \log \left(1-\alpha\right)\right] \\
&+& \sum \Delta y_{i}\left[\log \alpha - \log \left(1-\alpha\right)\right] + \sum_{i}\frac{1}{2}\left(\Delta y_{i}\right)^{2}\left[\frac{1}{\alpha} - \frac{1}{1-\alpha}\right] 
\end{eqnarray*}
The sum of $\Delta y_{i}$ becomes
\begin{eqnarray*}
\sum_{i} \Delta y_{i} =  o_{p}\left(\frac{1}{\sqrt{n}}\right)
 \end{eqnarray*}
and the sum of the square of $\Delta y_{i}$ becomes
\begin{eqnarray*}
& &\sum \left(\Delta y_{i}\right)^{2} = \left(1-\alpha\right)^{2} \xi^{2} + o_{p}\left(\frac{1}{\sqrt{n}}\right)
\end{eqnarray*}
Therefore, the entropy term becomes
\begin{eqnarray*}
& &\sum_{i}\left\{\hat{y_{i1}}\log{\hat{y_{i1}}} + (1-\hat{y_{i1}})\log{(1-\hat{y_{i1}})}\right\} \\
& &= n\left[\alpha \log \alpha + \left(1-\alpha\right) \log \left(1-\alpha\right)\right] + \frac{1-\alpha}{2\alpha} \xi^{2}
\end{eqnarray*}
The other terms can be calculated as follows:
\begin{eqnarray*}
& &\log{\frac{\Gamma (n+2\phi)}{\Gamma(n_{1}+\phi)\Gamma(n-n_{1}+\phi)}} \\
&=& \frac{1}{2}\log n - \left(n\alpha+\phi-\frac{1}{2}\right)\log \alpha \\
&-&\left(n(1-\alpha)+\phi - \frac{1}{2}\right)\log (1-\alpha) -\frac{1}{2}\log 2\pi+ o(1)
\end{eqnarray*}

\begin{eqnarray*}
& &\frac{1}{2}\log \left(1 + \sigma^{2}\sum y_{i1}\right) \\
&\simeq& \frac{1}{2}\left[\log n + \log \alpha + \log \sigma^{2} \right]+ \frac{1}{2}\log \left(1 + \frac{1}{n\alpha \sigma^{2}}\right) \nonumber \\
&=& \frac{1}{2}\left[\log n + \log \alpha + \log \sigma^{2} \right] + o(1)
\end{eqnarray*}

\begin{eqnarray*}
-\frac{1}{2}\frac{\left(\sum X_{i}\hat{y_{i1}}\right)^{2}}{\sum \hat{y_{i1}} + \frac{1}{\sigma^{2}}} &=& -\frac{1}{2}\langle b\rangle^{2}\left(\sum \hat{y_{i1}} + \frac{1}{\sigma^{2}}\right) \\
&=& -\frac{1}{2\alpha}\xi^{2} + o\left(1\right)
\end{eqnarray*}
By integrating them, we obtain the variational free energy
\begin{eqnarray}
F &=& \log n  +(1-\phi)\log \alpha- \left(\phi-\frac{1}{2}\right)\log (1-\alpha) + \frac{1}{2}\log \sigma^{2}  \nonumber \\
&-& \frac{1}{2}\xi^{2} + \log \frac{\Gamma(2\phi)}{\prod(\Gamma(\phi))} - \frac{1}{2}\log 2\pi + o(1)
\end{eqnarray}

From these results, we obtain
\begin{equation}
F - F_{0} = \log{n} + o(\log n),
\end{equation}
and the proof is completed.
\end{proof}

\subsection{Asymptotic form of $F$ when $\sum_{i}y_{i1}/n \rightarrow 0$}

\begin{theorem}
Let us define the function $f(y_{i})$ as,

\begin{eqnarray*}
f(y_{i}) &=& \sum_{i}\left\{\hat{y_{i1}} \log \hat{y_{i1}} + \left(1-\hat{y_{i1}}\right)\log \left(1-\hat{y_{i1}}\right)\right\}  \\
&-& \frac{\left(\sum_{i}X_{i}\hat{y_{i1}}\right)^{2}}{2\left(n_{1}+1/\sigma^{2}\right)}
\end{eqnarray*}
When we fix $\sum_{i}\hat{y_{i1}} = n_{1}$,  the minimum of $f(\hat{y_{i1}})$ satisfies
\begin{equation}
f(y_{i}) = -n_{1}\log n + n_{1} \log n_{1} - n_{1} + o(1)
\end{equation}
and $F-F_{0}$ becomes
\begin{equation}
F-F_{0} = \phi\log \frac{n}{n_{1}} + \log n_{1} + \mathcal{O}_{p}(1)
\end{equation}

\end{theorem}

\begin{proof}
In this case, the leading order of the logarithm of the ratio of the gamma function is different from the previous case. It becomes

\begin{eqnarray*}
& &\frac{\log{\Gamma(n+2\phi)}}{\prod \log{\Gamma(\sum_{i}\hat{y_{ik}})+\phi)}} = ( n_{1} + \phi)\log n \\
&-& (n_{1} + \phi - \frac{1}{2})\log (n_{1}+\phi) -(n-n_{1})\log \left(1-\frac{n_{1}}{n}\right) + \mathcal{O}(1) \nonumber \\
\end{eqnarray*}
where we define $\sum_{i}\hat{y_{i1}} \equiv n_{1}$.\\
Applying the method of Lagrange multipliers,  we minimize the function as follows:
\begin{eqnarray*}
f_{1}(\hat{y_{i1}}) &=& \sum_{i}\left\{\hat{y_{i1}} \log \hat{y_{i1}} + \left(1-\hat{y_{i1}}\right)\log \left(1-\hat{y_{i1}}\right)\right\} \\
&-& \frac{\left(\sum_{i}X_{i}\hat{y_{i1}}\right)^{2}}{2\left(n_{1}+1/\sigma^{2}\right)} - \lambda \left(\sum_{i}\hat{y_{i1}} - n_{1}\right)
\end{eqnarray*}
The equation of the stationary condition is given as
\begin{equation}
\frac{\partial f_{1}}{\partial \hat{y_{i1}}}= \log \frac{\hat{y_{i1}}}{1-\hat{y_{i1}}} - \frac{\sum_{j}X_{j}\hat{y_{j1}} X_{i}}{\left(n_{1} + 1/\sigma^{2}\right)} - \lambda = 0
\end{equation}
By solving it with $y_{i}$, we obtain
\begin{equation}
\hat{y_{i1}} = \frac{1}{1+ \exp(-A (X_{i}-B))}
\end{equation}
where
\begin{eqnarray*}
A &\equiv& \frac{\sum_{j}X_{j}\hat{y_{j1}}}{\left(n_{1} + 1/\sigma^{2}\right)} \\
B&\equiv& \frac{\lambda}{A}
\end{eqnarray*}
Let us assume that $A>0$, and $X_{1}\le X_{2}\le ... \le X_{n}$. This assumption does not lose the generality.
From the result of lemma 1 the proof of which is provided in the appendix, the asymptotic form of the trimmed sum of $X_{i}$ becomes
\begin{equation}
\sum_{i = n-n_{1}+1}^{n}X_{(i)n} \rightarrow n_{1}\sqrt{2\log \frac{n}{n_{1}}}
\end{equation}
where $X_{(i)n}$ means the $i$th order statistics, and $\left\{X_{(i)n}\right\}$ satisfies
\begin{equation*}
X_{(1)n} \le X_{(2)n} \le ... \le X_{(n)n}
\end{equation*}
Using this, we can obtain 
\begin{equation*}
0 < A \le \sqrt{2\log \frac{n}{n_{1}}}
\end{equation*}
As $\sum_{i}\hat{y_{i1}} = n_{1}$, and $\lim_{n\rightarrow \infty} \frac{n_{1}}{n} = 0$, the number of $\hat{y_{i1}}$ that satisfies $\hat{y_{i1}} > 1/2$ should not be $\mathcal{O}(n)$. 
Therefore, $B$ should go to $\infty$ when $n \rightarrow \infty$.\\
Let $Z$ be a constant that satisfies $Z \rightarrow \infty$ and $\frac{B}{Z} \rightarrow \infty$.\\
We will write the number of $\hat{y_{i1}}$ that satisfies $X_{i} < B/Z$ as $\alpha n_{1}$, and $\beta \equiv 1-\alpha$.
\begin{equation}
|X_{i}| \le B/Z \Rightarrow \hat{y_{i1}} = \frac{1}{1+ \exp(-A X_{i}+AB)} \sim \exp(-AB)
\end{equation}
Let us split $f(\hat{y_{i1}})$ into the three parts, 
\begin{equation}
f(\hat{y_{i1}}) = f_{1}(\hat{y_{i1}}) + f_{2}(\hat{y_{i1}}) + f_{3}(\hat{y_{i1}}),
\end{equation}
where
\begin{equation}
f_{1} = \sum_{X_{i}<B/Z}\hat{y_{i1}}\log \hat{y_{i1}} + (1-\hat{y_{i1}})\log (1-\hat{y_{i1}})
\end{equation}

\begin{equation}
f_{2} = \sum_{X_{i}>B/Z}\hat{y_{i1}}\log \hat{y_{i1}} + (1-\hat{y_{i1}})\log (1-\hat{y_{i1}})
\end{equation}

\begin{equation}
f_{3} = -\frac{\left(\sum_{i}X_{i}\hat{y_{i1}}\right)^{2}}{2\left(n_{1}+1/\sigma^{2}\right)}
\end{equation}
From the convexity, $f_{1}$ satisfies the following inequality:

\begin{eqnarray*}
f_{1} &\ge& n \left[\frac{\alpha n_{1}}{n}\log \frac{\alpha n_{1}}{n} + \left(1-\frac{\alpha n_{1}}{n}\right)\log \left(1- \frac{\alpha n_{1}}{n}\right)\right] \\
&=& -\alpha n_{1} \log n + \alpha n_{1} \log \alpha n_{1} - \alpha n_{1}+ \alpha^{2}\frac{n_{1}^{2}}{n}.
\end{eqnarray*}
Also, because the minimum of the function $g(y) = y \log y + (1-y)\log (1-y)$ is $g(y=1/2) = -\log 2$, $f_{2}$ satisfies
\begin{equation}
f_{2} = \sum_{X_{i}>B/Z} \hat{y_{i1}}\log \hat{y_{i1}} + (1-\hat{y_{i1}})\log (1-\hat{y_{i1}}) \ge -\beta n_{1}\log 2
\end{equation}
As for $f_{3}(\hat{y_{i1}})$, the term $\sum_{i}X_{i}\hat{y_{i1}}$ in the numerator satisfies

\begin{eqnarray*}
\sum_{i}X_{i}\hat{y_{i1}} &=& \sum_{X_{i}>B/Z}X_{i}\hat{y_{i1}} + \sum_{X_{i}<B/Z}X_{i}\hat{y_{i1}} \\
&\le& \left(\sum_{X_{i}>B/Z}X_{i} \right)+ \sum_{X_{i}<B/Z}X_{i}\hat{y_{i1}} \\
&=& \beta n_{1}\sqrt{2\log \frac{n}{\beta n_{1}}} + \sum_{i}X_{i}\alpha \frac{n_{1}}{n}
\end{eqnarray*}
In the last line, we use the result, 
\begin{equation}
\alpha n_{1} = n\exp \left(-AB\right) .
\end{equation}
Therefore,

\begin{equation}
\sum_{i}X_{i}\hat{y_{i1}} \le \alpha \frac{n_{1}}{n}\sum_{X_{i}<B/Z}X_{i} + \beta n_{1}\sqrt{2\log \frac{n}{\beta n_{1}}}
\end{equation}
The condition in which equality is satisfied is $\alpha=1, \beta=0$, and 
\begin{equation}
f_{3}(\hat{y_{i1}}) \ge -\frac{1}{2}\frac{\left(\alpha n_{1}\frac{1}{n}\sum_{X_{i}<B/Z} X_{i}\right)^{2}}{\left(n_{1} + 1/\sigma^{2}\right)}
\end{equation}
We can see that the $\alpha$ and $\beta$ that gives the maximum of $f_{1} + f_{2}$ under $\alpha + \beta = 1$ are also $\alpha=1, \beta=0$.\\
Therefore,  we obtain 
\begin{equation}
f(\hat{y_{i1}}) \ge -n_{1}\log n + n_{1} \log n_{1} - n_{1} + o(1)
\end{equation}
By adding the log gamma term to it, we obtain the minimum of the variational free energy, 
 \begin{equation}
F-F_{0} = \phi\log \frac{n}{n_{1}} + \log n_{1} + \mathcal{O}_{p}(1)
\end{equation}

\end{proof}
From the results of Theorem 1 and Thereom 2, we can obtain the asymptotic behavior of the variational free energy as a function of $\phi$ within the $\mathcal{O}(\log n)$ as
\[
  F - F_{0} = \left\{ \begin{array}{ll}
     \phi \log n + o\left(\log n\right)& (\phi < 1) \\
    \log n+ o\left(\log n\right)& (otherwise)
  \end{array} \right.
\]
This clearly shows that the phase transition exists in our model, and the critical value of the hyperparameter $\phi_{\rm{cr}}$ is $\phi_{\rm{cr}} = 1$.
Note that this result and the critical value are different from those obtained in the previous study\cite{WatanabeWatanabe2006}, because our model and theirs have different parameter space.  
 
We should also note that the configuration of $\left\{\hat{y_{i1}}\right\}$ of the solution is clearly different depending on the value of the hyperparameter. When the $\phi \ge 1$, the solution satisfies $\sum \hat{y_{i1}}\sim\mathcal{O}(1)$. This means that the sample is described under the A. H. by the two clusters that have a mixture ratio of the same order, but the mean of the one cluster may slightly deviate from the origin. The hypothesis test scheme based on this can be regarded testing whether the number of the cluster is one or not.
 
In contrast, when $\phi <1$,  the $\sum \hat{y_{i1}}$ obtained as the solution is small. This means that the vast majority of the sample is described under the A. H. by the one cluster whose mean is located in the origin. The other cluster may have an arbitrary mean, but the mixture ratio is very small. Under such circumstances, the hypothesis test scheme based on this can be regarded as testing for the existence of outliers.
 
Our result shows that we should choose an appropriate hyperparameter suitable for the purpose.

\section{Asymptotic form of the variational free energy on  the $\mathcal{O}(1)$}
In this section, we consider a situation in which it is difficult to distinguish whether or not a sample is generated from one cluster. One such a situation is that in which $n\langle a \rangle$ is $\mathcal{O}(n)$, but $\langle b \rangle $ is close to $0$. From the discussion in the previous section, this corresponds to the case in which $\phi >1$.

The following theorem on the asymptotic form of the variational free energy is derived under the above assumption.

\begin{theorem}
The variational free energy of the two component Gaussian mixture becomes 
\begin{eqnarray}
F &=& \log n  -\left(\phi -1 \right) \log \left(\phi -1 \right) - \left(\phi -\frac{1}{2} \right) \log \left(\phi -\frac{1}{2} \right) \nonumber \\ 
&+& \left(2\phi -\frac{3}{2} \right) + \frac{1}{2}\log \sigma^{2} - \frac{1}{2}\xi^{2} + \log \frac{\Gamma(2\phi)}{\prod(\Gamma(\phi))} \nonumber \\ 
&-& \frac{1}{2}\log 2\pi + o(1)
\end{eqnarray}
when the hyperparameter satisfies $\phi > 1$. 
Here, $\xi$  is a probabilistic variable that obeys $\xi \sim \mathcal{N}(0, 1^{2})$. 
\end{theorem}

\begin{proof}
As proven in Theorem 1,  the variational free energy becomes
\begin{eqnarray}
F &=& \log n  +(1-\phi)\log \alpha- \left(\phi-\frac{1}{2}\right)\log (1-\alpha) + \frac{1}{2}\log \sigma^{2}  \nonumber \\
&-& \frac{1}{2}\xi^{2} + \log \frac{\Gamma(2\phi)}{\prod(\Gamma(\phi))} - \frac{1}{2}\log 2\pi + o(1)
\end{eqnarray}
From the variational principle, $\alpha$ is determined as 
\begin{equation}
\alpha = \rm{argmin}\left[F(\alpha)\right] \equiv \alpha_{0}
\end{equation}
$\alpha_{0}$ is the solution of $\frac{d F}{d \alpha} = 0$, that is,
\begin{equation}
\alpha_{0} = \frac{\phi -1}{2\phi - \frac{3}{2}}
\end{equation}
By substituting this into $F$, we can obtain 
\begin{eqnarray}
F &=& \log n -\left(\phi -1 \right) \log \left(\phi -1 \right) - \left(\phi -\frac{1}{2} \right) \log \left(\phi -\frac{1}{2} \right) \nonumber \\
&+& \left(2\phi -\frac{3}{2} \right) + \frac{1}{2}\log \sigma^{2} - \frac{1}{2}\xi^{2} + \log \frac{\Gamma(2\phi)}{\prod(\Gamma(\phi))} \nonumber \\
&-& \frac{1}{2}\log 2\pi + o(1)
\end{eqnarray}
This is the result that we want to derive.\\
\end{proof}
In Figure \ref{fig:phi-alpha0}, $\alpha_{0}$  is plotted as a function of $\phi$.We can see that $\alpha_{0}$ shows the power-law behavior around the critical point $\phi_{\rm{cr}} = 1$, from the form of $\alpha_{0}$ derived above. That is, 
\begin{equation}
\alpha_{0} \sim \left(\phi-\phi_{\rm{cr}}\right)^{-1}
\end{equation}

\begin{figure}[H]
\begin{center}
\includegraphics[width=11cm]{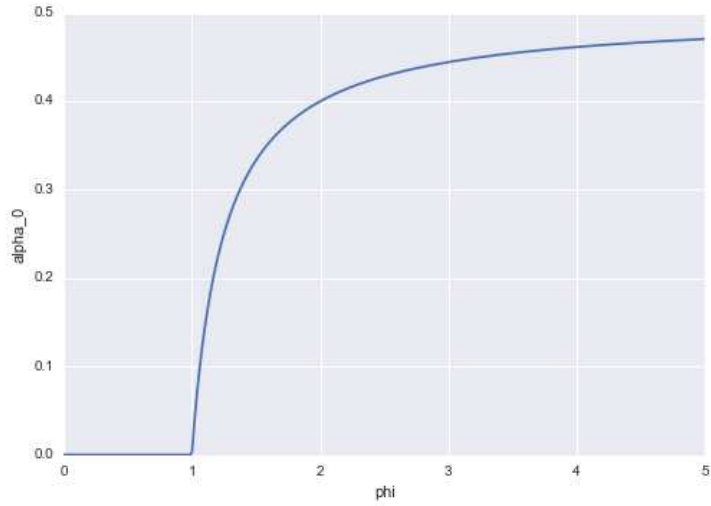}
\caption{Variational parameter $\alpha_{0}$ that minimizes variational free energy  $F$ as function of the hyperparameter $\phi$.}
\label{fig:phi-alpha0}
\end{center}
\end{figure}

Figure \ref{fig:phi-alpha0} shows that the stochastic behavior of the variational free energy is determined by $\xi$. Under the N. H., $\xi$ follows a standard normal distribution and the distribution of the variational free energy can be described by a $\chi^{2}$ distribution. The validity of these results are examined in the next section.

\newpage

\section{Numerical Experiment}

In this section, we show the result of our numerical experiments to examine the validity of our theoretical results.\\
First, to see the validity of the asymptote for a finite sample size, we compared the asymptote with one that is numerically calculated by an iterative algorithm, the variational Bayes-EM (VB-EM) algorithm.
We set the hyperparameter as a sufficiently large value, $\phi=20$, and calculated the asymptote in cases in which $n=200, 400,800,1600,3200,6400$ cases. To see the variance, we calculated them for $100$ different sample set. The results are shown in figure  \ref{fig:comparison-VB-asymptote}. We can see the theoretical asymptote and the numerical result  match well as a distribution. 

\begin{figure}[H]
\begin{center}
\includegraphics[width=11cm]{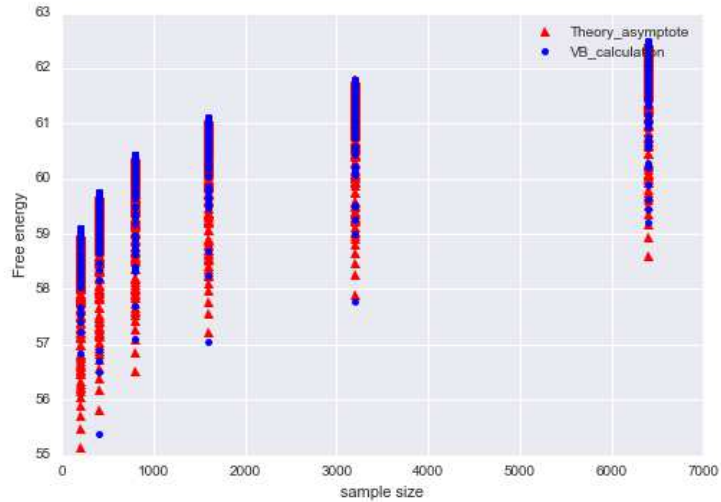}
\caption{Comparison of variational free energy calculated by VB-EM algorithm with the asymptote we derived for different sample sizes. Red triangles corresponds to the variational free energy calculated from theoretically derived asymptote, and blue circles corresponds to variational free energy numerically calculated by variational Bayes.}
\label{fig:comparison-VB-asymptote}
\end{center}
\end{figure}

We also numerically calculated the rejection rates for a finite sample with the VB-EM algorithm, and compared it with the threshold determined from the asymptote we derived.  We  numerically calculated the variational free energy for many sample sets independently generated from the null hypothesis, and determined the rejection rate as the ratio of the number of the sample set whose variational free energy becomes less than the threshold to the total number of the sample sets. 

Through the numerical experiments, the hyperparameter was set as $\phi = 20$ and we evaluated the variational free energy for the 5000 sample sets generated from the null hypothesis.  

The results are summarized in the table \ref{tab:rejection rates}. The results show that the threshold derived from the asymptote functions correctly. Therefore, we can conclude that the asymptotic form of the variational free energy we derived is valid.

\begin{table}[H]
\begin{center}

\caption{Rejection rates  calculated numerically by variational Bayes. Threshold is calculated from the asymptote of the variational free energy analytically derived in previous section.}

\vspace{5mm}
 
 \begin{tabular}{|c||c|r|r|} \hline
 \multicolumn{1}{|c||}{sample size}
 & \multicolumn{3}{|c|}{rejection rates} \\\hline 
 $n$ & 10\% & 5\% & 1\% \\\hline
 100 & 7.5\%& 3.6\%& 0.8\% \\\hline
 200 & 8.5\%&  4.3\%& 0.8\% \\\hline
 400 & 9.0\%& 4.4\%& 0.8\% \\\hline
 800 & 9.9\%& 5.1\%& 1.1\% \\\hline
 \end{tabular}
 \label{tab:rejection rates}
 \end{center}
\end{table}

From the results we have shown,  now the hypothesis test of homogeneity based on variational Bayes, which we refer it to as the VB test, can be constructed as follows.

First, calculate the variational free energy from the sample numerically by VB-EM algorithm. In this procedure, the hyperparameter $\phi$ should be set as greater than one.

Second, test whether the variational free energy is below the threshold or not, derived from the asymptote we derived in Section 5.  The stochastic behavior of the asymptote is described by the square of the standard normal distribution, and it is easy to calculate the threshold for the rejection rates one needs, by combining the well-known behavior of the $\chi^{2}$ distribution and the asymptote we derived.

\newpage

\section{Conclusion}

We discussed a new hypothesis test for the homogeneity using variational Bayes. We derived the variational free energy of the normal mixture model and showed that the phase transition occurs when the hyperparameter $\phi$ in the prior. exceeds the critical value $\phi > \phi_{cr} = 1$.  We also derived the analytical asymptote of the variational free energy on the $\mathcal{O}(1)$ in the $\phi > 1$ phase. This enabled us to construct a new approach to testing for homogeneity, the VB test.\\
The application of variational Bayes for hypothesis tests is not limited to the problem we discussed in this paper. As future problems, it would also be interesting to construct hypothesis tests for other singular models, using this framework.

\appendix
\section{Proof of the Lemma 1}
\begin{lemma}
Let $X_{1}, X_{2}, ...X_{n}$ be an i.i.d sample generated from the standard normal distribution $\mathcal{N}(0,1^{2})$, and let $X_{(i),n}$ be the order statistics of the sample, that satisfies $X_{(1),n} \le X_{(2),n} \le ... \le X_{(n),n}$.

Let us consider the trimmed sum of the largest $n_{1}$th data from the sample. When $n_{1} \rightarrow \infty$ and $n_{1}/n\rightarrow 0$, the asymptotic behavior of the sum is

\begin{equation}
S = \sum_{i=n-n_{1}+1}^{n}X_{(i),n} \rightarrow \sqrt{2\log \frac{n}{n_{1}}} + o_{p}(n_{1})
\end{equation}
\end{lemma}

\begin{proof of lemma}
As the normal distribution satisfies the von Mises conditions, the asymptote of the $n_{1}$th maximum values $x_{(n-n_{1}+1),n}$ satisfies
\begin{equation}
\left(X_{(n-n_{1}+1),n}-a_{n}\right)/b_{n}\rightarrow \mathcal{N}(0,1)
\end{equation}
where $a_{n} \equiv F^{-1}\left(1-\frac{n_{1}}{n}\right)$ and $b_{n} \equiv \sqrt{n_{1}}/\left(n f(a_{n})\right)$, here $F(x)$ means the cdf of $X_{i}$, and $f(x)$ means the distribution function of $X_{i}$ (see Theorem 8.3.4 and Theorem 8.5.3 in \cite{Arnold2008}).\\
In our case, the asymptotic form of $a_{n}$ becomes $a_{n} \rightarrow \sqrt{2\log {\frac{n}{n_{1}}} - \log \log \left(\frac{n}{n_{1}}\right)^{2}}$,  $b_{n} \rightarrow \frac{\sqrt{n_{1}}}{n}\frac{1}{\sqrt{2\pi}}e^{-\frac{1}{2}a_{n}^{2}}$, and the leading term of the $X_{(n-n_{1}+1),n} $becomes
\begin{equation}
X_{(n-n_{1}+1),n} = a_{n} + o_{p}(\frac{n_{1}}{n})
\end{equation} 
Let us proof the lemma using this result.
First, we split the $X_{(n-n_{1}+1),n}, ... X_{(n),n}$ samples by $T$ groups that satisfy $1 \ll T < n_{1} \ll n$.
Each group contains $\left[n_{1}/T\right]$ terms.\\
The maximum in the $t+1$th group, $Y_{t+1}$ satisfies

\begin{eqnarray*}
Y_{t+1} &\le& \sqrt{2\log \left(n/\left(n_{1}*t/T\right)\right)} + o_{p}(\frac{n_{1}}{n}) \\
&=&\sqrt{2\log \left(n/n_{1}\right) + 2\log \left(T/t\right)} + o_{p}(\frac{n_{1}}{n}) \\
&=&\sqrt{2\log \left(n/n_{1}\right)}  \times \sqrt{1 + \log{\left(T/t\right)}/\log{\left(n/n_{1}\right)}} + o_{p}(\frac{n_{1}}{n}) \\
&\le& \sqrt{2\log \left(n/n_{1}\right)}  \times \left(1 + \log{\left(T/t\right)}/\log{\left(n/n_{1}\right)}\right) \\
&=&  \sqrt{2\log \left(n/n_{1}\right)} + \sqrt{2} \log{\left(T/t\right)}/\sqrt{\log{\left(n/n_{1}\right)}}
\end{eqnarray*}
Therefore,
\begin{eqnarray*}
\sum_{i=n-n_{1}+1}^{n}x_{(i),n} &\le& \frac{n_{1}}{T}\sum_{t=1}^{T-1} Y_{t+1} + \frac{n_{1}}{T}\sqrt{2\log n} \\
(r.h.s)&=& \frac{n_{1}}{T}\left(T-1\right) \sqrt{2\log \left(n/n_{1}\right)} \\
&+& \sqrt{2}\left(T-1\right)/\sqrt{\log{\left(n/n_{1}\right)}}+ \frac{n_{1}}{T}\sqrt{2\log n} \\
&\le& \left(n_{1} - \frac{n_{1}}{T}\right)\sqrt{2\log \left(n/n_{1}\right)} \\
&+& \frac{n_{1}}{T}\left(\sqrt{2\log n/n_{1}} +\sqrt{2}\frac{\log n_{1}}{\sqrt{\log (n/n_{1)}}}\right) \\
&=& n_{1}\sqrt{2\log \left(n/n_{1}\right)} + \frac{n_{1}}{T}\frac{\sqrt{2} \log n_{1}}{\sqrt{\log \left(n/n_{1}\right)}} + o_{p}(n_{1})
\end{eqnarray*}
If we choose $T$ that satisfies $1 \ll T < n_{1} \ll n$ properly, e.g., $T = \sqrt{n_{1}}$, the second term becomes $o_{p}(n_{1})$.
and the lemma is proven.
\end{proof of lemma}

\bibliographystyle{unsrtnat}

\bibliography{list.bib}
\end{document}